\documentclass[a4paper,UKenglish,numberwithinsect]{lipics_md}
\usepackage[normalem]{ulem}
\usepackage{graphicx}
\usepackage{amsmath}
\usepackage{amssymb}
\usepackage{xspace}
\usepackage{epsfig}
\usepackage{soul}

\usepackage[mathlines]{lineno}

\usepackage{etoolbox} 

\newcommand*\linenomathpatchAMS[1]{%
	\expandafter\pretocmd\csname #1\endcsname {\linenomathAMS}{}{}%
	\expandafter\pretocmd\csname #1*\endcsname{\linenomathAMS}{}{}%
	\expandafter\apptocmd\csname end#1\endcsname {\endlinenomath}{}{}%
	\expandafter\apptocmd\csname end#1*\endcsname{\endlinenomath}{}{}%
}

\expandafter\ifx\linenomath\linenomathWithnumbers
\let\linenomathAMS\linenomathWithnumbers
\patchcmd\linenomathAMS{\advance\postdisplaypenalty\linenopenalty}{}{}{}
\else
\let\linenomathAMS\linenomathNonumbers
\fi

\newcommand {\mm}[1] {\ifmmode{#1}\else{\mbox{\(#1\)}}\fi}

\newcommand{\ignore}[1]{}

\newsavebox{\smallProofsym}                            
\savebox{\smallProofsym}                               %
{
\begin{picture}(6,6)                                   %
\put(0,0){\framebox(6,6){}}                            %
\put(0,2){\framebox(4,4){}}                            %
\end{picture}                                          %
}                                                      %
\makeatletter
\long\def\@makecaption#1#2{%
  \vskip\abovecaptionskip
  \sbox\@tempboxa{\small #1: #2}%
  \ifdim \wd\@tempboxa >\hsize
    \small #1: #2\par
  \else
    \global \@minipagefalse
    \hb@xt@\hsize{\hfil\box\@tempboxa\hfil}%
  \fi
  \vskip\belowcaptionskip}
\makeatother

\newcommand{\Grass}[2] {\mm{{\mathbb Gr}_{{#1},{#2}}}}
\newcommand{\Stiefel}[2] {\mm{{\mathbb St}_{{#1},{#2}}}}

\newcommand{\Indicator}[1]  {\mm{{\textbf{1}}_{#1}}}

\newcommand{\Rspace}        {\mm{{\mathbb R}}}

\newcommand{\Bcal}[2]       {\mm{\mathcal B}_{#1}{({#2})}}

\newcommand{\tile}[2]       {\mm{{J}{({#1},{#2})}}}
\newcommand{\DConstant}[2]  {\mm{{\cal D}_{{#1},{#2}}}}

\newcommand{\moment}[3]     {\mm{{\bf m}_{{#1},{#2}}^{({#3})}}}
\newcommand{\cScape}[2] {\mm{{\mathcal V}_{{#1}}{({#2})}}}

\newcommand{\VolPower}[3]    {\mm{{\|{#2}\|}_{#1}^{#3}}}
\newcommand{\Vol}[2]    {\mm{{\|{#2}\|}_{#1}}}

\newcommand{\Angle}[2]      {\mm{\varphi}{({#1},{#2})}}
\newcommand{\proj}[2]       {\mm{{#1}|_{#2}}}

\newcommand{\Voronoi}[1]  {\mm{{\rm Vor}{({#1})}}}
\newcommand{\Delaunay}[1] {\mm{{\rm Del}{({#1})}}}

\newcommand{\dime}[1]       {\mm{\rm dim\,}{#1}}
\newcommand{\conv}[1]       {\mm{\rm conv}{{#1}}}

\newcommand{\affine}[1]     {\mm{\rm aff\,}{#1}}

\newcommand{\Edist}[2]      {\mm{\|{#1}-{#2}\|}}

\newcommand{\diff}          {\mm{\rm \,d}}

\newcommand{\Skip}[1]       {}

\title{Average and Expected Distortion of Voronoi Paths and Scapes\footnote{This project has received funding from the European Research Council (ERC) under the European Union's Horizon 2020 research and innovation programme, grant no.\ 788183, from the Wittgenstein Prize, Austrian Science Fund (FWF), grant no.\ Z 342-N31, and from the DFG Collaborative Research Center TRR 109, `Discretization in Geometry and Dynamics', Austrian Science Fund (FWF), grant no.\ I 02979-N35.}
}
\titlerunning{Average and Expected Distortion}

\author[]{Herbert Edelsbrunner}
\author[]{Anton Nikitenko}
\affil[]{IST Austria (Institute of Science and Technology Austria), Am Campus 1, \\ 3400 Klosterneuburg, Austria, \texttt{edels@ist.ac.at}, \texttt{anton.nikitenko@ist.ac.at}}

\authorrunning{H. Edelsbrunner and A. Nikitenko}

\mathsubjclass{60D05 Geometric probability and stochastic geometry, 68U05 Computer graphics; computational geometry.} 


\begin{document}
\maketitle

\begin{abstract}
  The approximation of a circle with the edges of a fine square grid distorts the perimeter by a factor about $\tfrac{4}{\pi}$.
  We prove that this factor is the same \emph{on average} (in the ergodic sense) for approximations of any rectifiable curve by the edges of any non-exotic Delaunay mosaic (known as \emph{Voronoi path}), and extend the results to all dimensions, generalizing Voronoi paths to \emph{Voronoi scapes}.
\end{abstract}

\section{Introduction}
\label{sec:1}

Given a locally finite set $A \subseteq \Rspace^d$ and a line segment, the \emph{Voronoi path} of the line segment is the dual of the Voronoi tessellation of $A$ intersected with the segment.
In other words, it consists of all Delaunay edges dual to Voronoi cells of dimension $d-1$ crossed by the line segment.
We generalize it to the \emph{Voronoi scape} of $A$ and a $p$-dimensional set $\Omega \subseteq \Rspace^d$, which is a multiset of the cells in the Delaunay mosaic of $A$.
In the generic case, when $\Omega$ intersects a Voronoi $(d-p)$-cell in a finite number of points, $\mu$, the Voronoi scape contains the corresponding Delaunay $p$-cell $\mu$ times.
We are interested in the \emph{distortion}, which is the ratio of the $p$-dimensional volume of the Voronoi scape over the $p$-dimensional volume of $\Omega$.

\smallskip
Considering the Voronoi tessellation of a stationary Poisson point process and a line segment in $\Rspace^2$, \cite{BTZ00} proves that the expected distortion is $\frac{4}{\pi}$.
Extending this work to $d > 2$ dimensions, \cite{dCDe18} proves that the expected distortion is $\sqrt{2d/\pi} + O(1/\sqrt{d})$.
We remove the ambiguity in this answer by proving that the expected distortion in $\Rspace^d$ is $d!!/(d-1)!!$, if $d$ is odd, and $\frac{2}{\pi} d!! / (d-1)!!$, if $d$ is even, in which $!!$ is the double factorial.
Furthermore, we generalize the result from the line segment to rectifiable $p$-dimensional sets and prove that the expected distortion is the binomial coefficient $\tbinom{d/2}{p/2}$, in which non-integer parameters
are understood in the way the Gamma function extends the factorial:
\begin{align}
    \DConstant{p}{d}  &=  \binom{d/2}{p/2}
      =  \frac{\Gamma(\frac{d}{2} + 1)}{ \Gamma(\frac{p}{2} + 1)
         \, \Gamma(\frac{d-p}{2} + 1)} \renewcommand{\arraystretch}{1.5}
      =  \left\{ \begin{array}{ll}
           \frac{d!!}{p!! \, (d-p)!!} \frac{2}{\pi}
             &  \mbox{\rm if $d$ is even and $p$ is odd} , \\
           \frac{d!!}{p!! \, (d-p)!!}  &  \mbox{\rm otherwise}.
       \end{array} \right.
    \label{eqn:DConstant}
\end{align} 
The binomial interpretation also provides the asymptotics for $\DConstant{p}{d}$;
for the values in small dimensions see Table~\ref{tbl:distortion}.
More precisely, we prove that \eqref{eqn:DConstant} is the \emph{average distortion} for sufficiently regular $p$-dimensional sets and Voronoi tessellations, in which the average is taken over all rigid motions of the set.
The claim for stationary Poisson point processes follows because they are invariant under rotations and translations.
The proof is based on a decomposition of $\Rspace^d \times \Grass{p}{d}$ related to the \emph{mixed complex} introduced in \cite{Ede99}.
As a byproduct, we get an expression for the volumes of the cells in the mixed complex; see Corollary~\ref{cor:mixed_volumes}.
{\renewcommand{\arraystretch}{1.6}
\begin{table}[!hbt]%
  \centering \footnotesize
  \begin{tabular}{c||cccccccccc}
  & $d = 1$ & $2$ & $3$ & $4$ & $5$ & $6$ & $7$ & $8$ & $9$ & $10$ \\
 \hline \hline
 $p = 1$ & $1$ & $\frac{4}{\pi }$ & $\frac{3}{2}$ & $\frac{16}{3 \pi }$ & $\frac{15}{8}$ & $\frac{32}{5 \pi }$ & $\frac{35}{16}$ & $\frac{256}{35 \pi }$ & $\frac{315}{128}$ & $\frac{512}{63 \pi }$ \\
 $2$ &  & $1$ & $\frac{3}{2}$ & $2$ & $\frac{5}{2}$ & $3$ & $\frac{7}{2}$ & $4$ & $\frac{9}{2}$ & $5$ \\
 $3$ &  &  & $1$ & $\frac{16}{3 \pi }$ & $\frac{5}{2}$ & $\frac{32}{3 \pi }$ & $\frac{35}{8}$ & $\frac{256}{15 \pi }$ & $\frac{105}{16}$ & $\frac{512}{21 \pi }$ \\
 $4$ &  &  &  & $1$ & $\frac{15}{8}$ & $3$ & $\frac{35}{8}$ & $6$ & $\frac{63}{8}$ & $10$ \\
 $5$ &  &  &  &  & $1$ & $\frac{32}{5 \pi }$ & $\frac{7}{2}$ & $\frac{256}{15 \pi }$ & $\frac{63}{8}$ & $\frac{512}{15 \pi }$ \\
 $6$ &  &  &  &  &  & $1$ & $\frac{35}{16}$ & $4$ & $\frac{105}{16}$ & $10$ \\
 $7$ &  &  &  &  &  &  & $1$ & $\frac{256}{35 \pi }$ & $\frac{9}{2}$ & $\frac{512}{21 \pi }$ \\
 $8$ &  &  &  &  &  &  &  & $1$ & $\frac{315}{128}$ & $5$ \\
 $9$ &  &  &  &  &  &  &  &  & $1$ & $\frac{512}{63 \pi }$ \\
 $10$ &  &  &  &  &  &  &  &  &  & $1$ 
\end{tabular}
  \caption{The average, resp.\ expected distortion in small dimensions.
  Note that even rows and columns form the Pascal triangle.}
  \label{tbl:distortion}
\end{table}}

\medskip \noindent \textbf{Outline.}
Section~\ref{sec:2} prepares the proof of our main result by computing the first and second moments of the $p$-dimensional volume of the projection of a unit $p$-cube in $\Rspace^d$.
Section~\ref{sec:3} studies the space of point-direction pairs.
Section~\ref{sec:4} introduces a mild regularity condition for Voronoi tessellations.
Section~\ref{sec:5} computes the volume of the cells in the mixed complex.
Section~\ref{sec:6} proves that $\DConstant{p}{d}$ is the average distortion for $p$-dimensional shapes in $\Rspace^d$, and the expected distortion if the tessellation is of a stationary Poisson point process.
Section~\ref{sec:8} concludes the paper.

\section{Random Projections}
\label{sec:2}

We need some preliminary computations.
Let $\Grass{p}{d}$ be the \emph{(linear) Grassmannian manifold}, whose points are the $p$-planes that pass through the origin in $\Rspace^d$.
Given a $p$-dimensional unit cube, $E \subseteq \Rspace^d$, and a $p$-plane, $L \in \Grass{p}{d}$, we write $\proj{E}{L}$ for the projection of the cube onto the plane, and $\Vol{p}{\proj{E}{L}}{}$ for its $p$-dimensional volume.
The \emph{$j$-th projection moment} is the average $j$-th power of the volume of the projection.
We express this moment as an integral over the Grassmannian equipped with the uniform probability measure in \eqref{eqn:momentDef1} and convert it to two equivalent expressions involving the angle to a fixed plane in \eqref{eqn:momentDef2} and \eqref{eqn:momentDef3}:
\begin{align}
  \moment{p}{d}{j}  &=  \int\nolimits_{L \in \Grass{p}{d}} \VolPower{p}{\proj{E}{L}}{j} \diff L ,
	\label{eqn:momentDef1} \\
  &=  \int\nolimits_{L \in \Grass{p}{d}} \cos^j \Angle{L}{L_0} \diff L 
	\label{eqn:momentDef2} \\
  &=  \int\nolimits_{F \in \Stiefel{p}{d}} \VolPower{p}{\proj{F}{L_0}}{j} \diff F .
	\label{eqn:momentDef3}
\end{align}
To explain \eqref{eqn:momentDef2} and \eqref{eqn:momentDef3}, we fix the plane $L_0 \in \Grass{p}{d}$ containing $E$. 
The \emph{angle} between two $p$-planes, $\Angle{L}{L_0} \in [0, \frac{\pi}{2} ]$, is defined as the arc-cosine of the ratio of $\Vol{p}{\proj{B}{L}}{}$ over $\Vol{p}{B}{}$ for any compact set with non-empty interior, $B \subseteq L_0$.
The angle is symmetric, so we can instead consider the integrand in \eqref{eqn:momentDef2} as the projection of a unit $p$-cube in a random $p$-plane onto $L_0$.
Formally, we write $\Stiefel{p}{d}$ for the \emph{Stiefel manifold} of orthonormal $p$-frames in $\Rspace^d$, we identify a frame with the unit $p$-cube it spans, and we integrate using the uniform probability measure of $\Stiefel{p}{d}$ to arrive at \eqref{eqn:momentDef3}.

\medskip
By construction, the $0$-th projection moment is equal to $1$, independent of $p$ and $d$.
We compute the $1$-st and $2$-nd projection moments, which curiously both have intuitive geometric interpretations.
\begin{lemma}[Projection Moments]
  \label{lem:projection_moments}
  Let $d \geq 0$ and $0 \leq p \leq d$.
  Then 
  \begin{align}
    \moment{p}{d}{1}  &=  \frac{\Gamma (\frac{p+1}{2}) \; \Gamma (\frac{d-p+1}{2})}
	{\Gamma (\frac{1}{2}) \; \Gamma (\frac{d+1}{2})} ,
        \label{eqn:moment1} \\
    \moment{p}{d}{2}  &=  1 / \binom{d}{p} =  \frac{p! \, (d-p)!}{d!}.
        \label{eqn:moment2}
  \end{align}
\end{lemma}
\begin{proof}
  The $1$-st projection moment appears in the classic Crofton formula of integral geometry, which says that the volume of a convex body is proportional to the average volume of its orthogonal projections.
  The constant of proportionality given in \eqref{eqn:moment1} can be found in \cite[Formula (5.8)]{ScWe08}.
  We use \eqref{eqn:momentDef3} together with a generalization of the Pythagorean theorem to compute the $2$-nd moment.
  By Pythagoras, the squared length of a line segment is the sum of squared lengths of its projections onto the coordinate axes.
  The Cauchy--Binet formula \cite[\S4.6]{BrWi89} can be used to generalize this to the squared volume of a $p$-dimensional parallelepiped in $\Rspace^d$.
  Let $P$ be such a parallelepiped, and write $P_i$ for its projection
  onto the $i$-th coordinate $p$-plane (in which the numbering is arbitrary).
  There are $\binom{d}{p}$ coordinate $p$-planes, and the Cauchy--Binet formula asserts
  \begin{align}
    \VolPower{p}{P}{2}  &=  \sum\nolimits_{i=1}^{\binom{d}{p}} \VolPower{p}{P_i}{2} .
    \label{eqn:CauchyBinet}
  \end{align}
  Letting $P = F \in \Stiefel{p}{d}$ be the uniformly random unit $p$-cube, we can take the expectation on both sides of \eqref{eqn:CauchyBinet}.
  We get $1$ on the left-hand side and the sum of $\binom{d}{p}$ identical terms on the right-hand side.
  Hence, the average squared $p$-dimensional volume of the projection is $1 / \binom{d}{p}$, as claimed.
\end{proof}

We set $\DConstant{p}{d} = \moment{p}{d}{1} / \moment{p}{d}{2}$ and leave it to the reader to verify that this agrees with \eqref{eqn:DConstant}, where $\DConstant{p}{d}$ is given in terms of Gamma functions as well as double factorials.

\section{Tiling the Space of Point-Directions} 
\label{sec:3}

We use the Delaunay mosaic to tile the space of \emph{point-direction pairs}, $\Rspace^d \times \Grass{p}{d}$.
Given a Delaunay mosaic of a set $A \subseteq \Rspace^d$, denoted $\Delaunay{A}$, consider a $p$-dimensional cell, $\gamma \in \Delaunay{A}$, and its dual $(d-p)$-dimensional Voronoi cell, $\gamma^* \in \Voronoi{A}$.
We define the \emph{$p$-tile} of $\gamma$ to consist of all pairs $(x, L) \in \Rspace^d \times \Grass{p}{d}$ such that $L+x$ has a non-empty intersection with $\gamma^*$, and $x$ lies in the projection of $\gamma$ onto $L+x$:
\begin{align}
  \tile{\gamma}{\gamma^*}  &=  \{ (x,L) \in \Rspace^d \times \Grass{p}{d}
    \mid  x \in \proj{\gamma}{L+x}
          \mbox{\rm ~and~} (L+x) \cap \gamma^* \neq \emptyset \} .
\end{align}
The \emph{tiles} decompose the space $\Rspace^d \times \Grass{p}{d}$ in the sense that they cover the space {while} their interiors are pairwise disjoint.
Since the detailed analysis of the boundaries is irrelevant for the current work, we only prove a weaker statement.
\begin{lemma}[Uniqueness of Tile]
  \label{lem:uniqueness_of_tile}
  Let $A \subseteq \Rspace^d$ be locally finite with $\conv{A} = \Rspace^d$, and let $0 \leq p \leq d$.
  Then for almost every point-direction pair,
  $(x, L) \in \Rspace^d \times \Grass{p}{d}$,
  there exists a unique $p$-tile, $\tile{\gamma}{\gamma^*}$, that contains $(x, L)$.
\end{lemma}
\begin{proof}
  Take any point-direction pair, $(x, L)$.
  Assume without loss of generality that $x = 0$ is the origin and $L = \Rspace^p$ is a coordinate $p$-plane in $\Rspace^d$.
  Map each point $a \in A$ to the point $a' = \proj{a}{L} \in \Rspace^p$, and let $a'' = - \Edist{a}{a'}^2 \in \Rspace$ be its \emph{weight}.
  The weighted points define a weighted Voronoi tessellation and the corresponding weighted Delaunay mosaic; see e.g.\ \cite{Aur87b}.
  The mosaic is generically a simplicial complex and generally a polyhedral complex, which is geometrically realized in $\Rspace^p$ by drawing each cell, $\gamma$, as the convex hull of the points that generate the $p$-dimensional Voronoi cells sharing $\gamma^*$.
  Consider the cells in $\Voronoi{A}$ that have a non-empty intersection with $L$, write $\cScape{L}{A}$ for the collection of dual cells in $\Delaunay{A}$, and observe that $\cScape{L}{A}$ is the Voronoi scape of $L$.

  As proved in \cite{Sib80}, the weighted Voronoi tessellation is the intersection of $L$ with $\Voronoi{A}$ and, by duality, the weighted Delaunay mosaic is the orthogonal projection of $\cScape{L}{A}$ to $L$.
  If $L$ and $\Delaunay{A}$ are in general position, then all Delaunay cells in $\cScape{L}{A}$ project injectively to $L$, and the cells of dimension less than $p$ form a set of zero measure.
  If $\Delaunay{A}$ covers $\Rspace^d$, then the weighted Delaunay mosaic covers $\Rspace^p$.
  Hence, for almost all point-direction pairs, $(x, L)$, there is a unique Delaunay $p$-cell $\gamma$, such that $(x, L) \in \tile{\gamma}{\gamma^*}$, as claimed.
\end{proof}

The proof of the lemma gives some insight into the motivation for choosing this particular tiling of the space of point-direction pairs.
We now compute the measure of a tile.
\begin{lemma}[Volume of Tile]
  \label{lem:volume_of_tile}
  The measure of $J = \tile{\gamma}{\gamma^*}$ is
  $\Vol{}{J}{} = \Vol{p}{\gamma}{} \, \Vol{d-p}{\gamma^*}{} / \binom{d}{p}$.
\end{lemma}
\begin{proof}
  The measure of the tile is the integral of $1$ over its pairs.
  Setting $x = y + z$, in which $y \in L$ and $z \in L^\bot$, the integral is
  \begin{align}
    \Vol{}{J}{}  &=  \int\nolimits_{L \in \Grass{p}{d}}
        \int\nolimits_{y \in L} \Indicator{y \in \proj{\gamma}{L}}
        {\int\nolimits_{z \in L^{\bot}}} \Indicator{(L+z) \cap \gamma^* \neq \emptyset}
             \diff z \diff y \diff L
    \label{eqn:tilemusic1} \\
                 &=  \Vol{p}{\gamma}{} \, \Vol{d-p}{\gamma^*}{} \,
        \int\nolimits_{L \in \Grass{p}{d}} \cos^2 \Angle{L}{\gamma} \diff L ,
    \label{eqn:tilemusic2}
  \end{align}
  where we get \eqref{eqn:tilemusic2} by noticing that the innermost integral in \eqref{eqn:tilemusic1} is the $(d-p)$-dimensional volume of the projection of $\gamma^*$ to $L^\bot$, which is $\Vol{d-p}{\gamma^*}{} \cos \Angle{L^\bot}{\gamma^*} = \Vol{d-p}{\gamma^*}{} \cos \Angle{L}{\gamma}$, and the middle integral is the $p$-dimensional volume of the projection of $\gamma$ to $L$, which is $\Vol{p}{\gamma}{} \cos \Angle{L}{\gamma}$.
  {Using \eqref{eqn:momentDef2}, we see that the integral in \eqref{eqn:tilemusic2} is $\moment{p}{d}{2}$, and using \eqref{eqn:moment2}, we get the claimed equation.}
\end{proof} 

We take a closer look at the projection of a tile to $\Rspace^d$.
Let $(x, L)$ be a point-direction pair in $J = \tile{\gamma}{\gamma^*}$ with $\dime{\gamma} = p$.
There are points $u \in \gamma$ and $v \in \gamma^*$ such that $x = \proj{u}{L+x}$ and $v = (L+x) \cap \gamma^*$.
Because of the right angle between the direction and the projection, we have $\Edist{x}{u}^2 + \Edist{x}{v}^2 = \Edist{u}{v}^2$, so $x$ lies on the smallest sphere that passes through $u$ and $v$.
Indeed, $u$ and $v$ define a $(d-1)$-dimensional set of point-direction pairs, and the points of these pairs all lie on the mentioned sphere.

\smallskip
Let $z_1 = \affine{\gamma} \cap \affine{\gamma^*}$ and observe that the sphere defined by $u$ and $v$ also passes through $z_1$.
{Let $R_0$ be the maximum distance between a point of $\gamma$ and a point of $\gamma^*$, and note that $R_0$ is the radius of every largest sphere that passes through the vertices of $\gamma$ and does not enclose any of the points in $A$; see Figure~\ref{fig:tile}.
A sphere with the latter property is commonly called an \emph{empty sphere} of $A$.}
{Since the diameter of the sphere spanned by $u$ and $v$ is $\Edist{u}{v} \leq R_0$, it follows that the ball with center $z_1$ and radius $R_0$ contains this sphere and thus the projection of $J = \tile{\gamma}{\gamma^*}$ to $\Rspace^d$; see again Figure~\ref{fig:tile}.
Hence, the volume of the projection of $J$ is at most $R_0^d$ times the volume of the unit ball in $\Rspace^d$.
Since we assume the uniform probability measure on $\Grass{p}{d}$, the same upper bound holds for the measure of $J$ itself.}
\begin{figure}[hbt]
  \centering \vspace{0.0in}
  \resizebox{!}{1.5in}{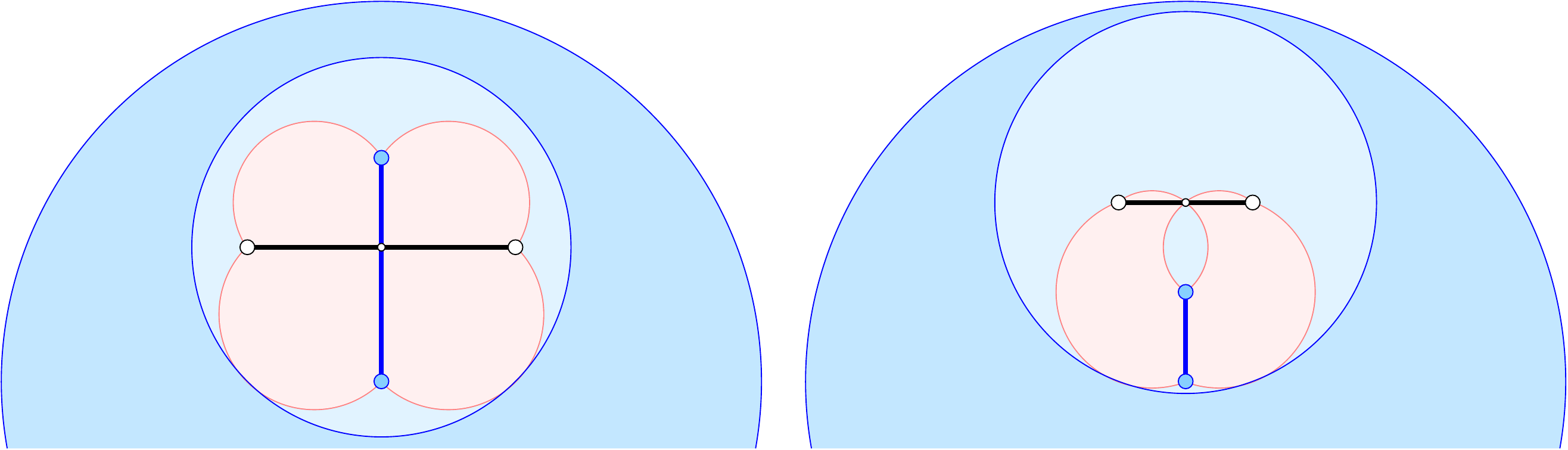}
  \vspace{-0.0in}
  \caption{{\footnotesize On the \emph{left}, the (\emph{pink}) projection of the tile defined by a Delaunay edge, $\gamma$, and its dual Voronoi edge, $\gamma^*$, has the topology of a disk, while on the \emph{right}, its projection has the topology of a pinched annulus.
  In both cases, it is contained in the disk with radius $R_0$ centered at $z_1$, and this disk and therefore also the projection of the tile is contained in the disk with radius $2 R_0$ centered at $z_2$.}}
  \label{fig:tile}
\end{figure}

{A weaker bound on this measure implied by a different ball will be more convenient.
Consider therefore the largest empty sphere that passes through the vertices of $\gamma$.
Its radius is $R_0$ and its center, $z_2$, lies on $\gamma^*$.
Hence $\Edist{z_2}{z_1} \leq R_0$, which implies that the ball with center $z_2$ and radius $2R_0$ contains the ball with center $z_1$ and radius $R_0$ and therefore also the projection of $J$ to $\Rspace^d$; see again Figure~\ref{fig:tile}.
We state the result for later reference.}
\begin{lemma}[Projection of Tile]
  \label{lem:projection_of_tile}
  {Let $z_2$ and $R_0$ be the center and radius of the largest empty sphere that passes through the vertices of $\gamma \in \Delaunay{A}$.
  Then the ball with center $z_2$ and radius $2R_0$ contains the projection of $J = \tile{\gamma}{\gamma^*}$ to $\Rspace^d$.}
\end{lemma}

\section{Mixed Regularity} 
\label{sec:4}

Taking the union of progressively more tiles, we eventually cover all of $\Rspace^d \times \Grass{p}{d}$.
However, at each step {during this construction}, some of the points miss some of the directions, and which directions are covered depends on the mosaic.
In what follows, we require a mild regularity condition for this tiling.
For a set $\Omega \subseteq \Rspace^d$, we call a tile a \emph{boundary tile} of $\Omega$ if its projection to $\Rspace^d$ contains at least one point inside and at least one point outside $\Omega$.
\begin{definition}[Mixed Regularity]
  \label{def:mixed_regularity}
  Let $A \subseteq \Rspace^d$ be locally finite.
  We say that $A$ has the property of \emph{mixed regularity} if, for any $p$, the total measure of the boundary $p$-tiles of a $d$-ball of radius $R$ centered at the origin is $o(R^d)$.
\end{definition}

Note that $\conv{A} = \Rspace^d$ is necessary for $A$ to have the mixed regularity property.
Indeed, if $\conv{A}$ does not cover $\Rspace^d$, then there exists an unbounded Voronoi cell and thus a tile with infinite measure.
Motivated by the analysis in Section~\ref{sec:3}, we give some sufficient conditions for a set $A \subseteq \Rspace^d$ to have the mixed regularity property:
\begin{lemma}[Sufficient Conditions]
  \label{lem:sufficient_conditions}
  A locally finite set $A \subseteq \Rspace^d$ has the mixed regularity property if one of the following holds:
  \begin{enumerate}
    \item the radii of all circumspheres of top-dimensional Delaunay cells are bounded;
    \item each ball in $\Rspace^d$ of radius greater than {some finite} $R_0$ contains a point of $A$;
    \item there is a function $g(R) = o(R)$ such that every ball of radius $g(R)$ that intersects the $d$-ball of radius $R$ centered at the origin contains at least one point of $A$.
  \end{enumerate}
\end{lemma}

Conditions~1 and 2 are equivalent, while Condition~3 is weaker.
We finish this section with an application to Poisson point processes:
\begin{lemma}[Mixed Regularity in Expectation]
  \label{lem:mixed_regularity_in_expectation}
  A stationary Poisson point process, $A \subseteq \Rspace^d$, has the mixed regularity property in expectation; that is: the total \emph{expected} measure of the boundary tiles of a $d$-ball
  with radius $R$ centered at the origin is $o(R^d)$.
\end{lemma}

\begin{proof}
  {Let $B(R)$ be the ball with radius $R$ centered at the origin, and let $J = \tile{\gamma}{\gamma^*}$ be a boundary tile.
  Its Delaunay cell, $\gamma$, is almost surely a simplex.
  Consider the top-dimensional cell that contains $\gamma$ as a face and whose circumsphere is the largest empty sphere that passes through the vertices of $\gamma$.
  Letting $z_2$ and $R_0$ be the center and radius of this sphere, Lemma~\ref{lem:projection_of_tile} implies that the concentric ball with twice the radius, $2R_0$, contains the projection of $J$ to $\Rspace^d$ and thus intersects the boundary of $B(R)$.
  \cite[Appendix A]{ENR17} studies the total number of such balls (albeit without doubling the radius), and it is straightforward to modify the proof to take the volume and doubling of the radius into account.
  With that, we get that the expected total volume of such balls containing the boundary tiles is $o(R^d)$.
  This implies the same upper bound for the expected total measure of the boundary tiles.}
\end{proof}

\section{Mixed Cells} 
\label{sec:5}

Call $\Vol{p}{\gamma} \, \Vol{d-p}{\gamma^*}$ the \emph{mixed cell volume} of a $p$-cell $\gamma \in \Delaunay{A}$ and its dual $(d-p)$-cell $\gamma^* \in \Voronoi{A}$.
This concept relates to a particular decomposition of $\Rspace^d$, as we now explain.
Given $A \subseteq \Rspace^d$, the $d$-dimensional cells of the \emph{mixed complex} defined in \cite{Ede99} are translates of the products $\frac{1}{2} \gamma \times \frac{1}{2} \gamma^*$.
{We refer to $\frac{1}{2} \gamma \times \frac{1}{2} \gamma^*$ as a \emph{mixed cell} and note that its volume is $\Vol{d}{\frac{1}{2} \gamma \times \frac{1}{2} \gamma^*}{} = \Vol{p}{\gamma}{} \, \Vol{d-p}{\gamma^*}{} / 2^d$.
As proved in \cite{Ede99}, the mixed cells have pairwise disjoint interiors and they cover $\Rspace^d$.}
Assuming the mixed regularity property, this implies that, up to a lower order term,
the cells for $p = 0$ cover a fraction of $1 / 2^d$ of $B(R)$.
By symmetry, this is also true for $p = d$.
We continue with a generalization of these bounds to dimension $p$ between $0$ and $d$.
\begin{corollary}[Mixed Cell Volumes]
  \label{cor:mixed_volumes}
  Let $A \subseteq \Rspace^d$ have the mixed regularity property.
  For any $0 \leq p \leq d$, the sum of the mixed cell volumes, over all $p$-cells of $\Delaunay{A}$ contained in a ball of radius $R$, is
  $\tbinom{d}{p} R^d \nu_d + o(R^d)$,
  in which $\nu_d$ is the volume of the unit ball in $\Rspace^d$.
\end{corollary}
\begin{proof}
  {Recall that $B(R)$ is the ball with radius $R$ centered at the origin.
  Set $\Bcal{p}{R} = B(R) \times \Grass{p}{d}$, let $M_p (R)$ be the smallest union of $p$-tiles that
  contains $\Bcal{p}{R}$, and let $\partial M_p(R)$ be the union of boundary tiles of $B(R)$.}
  Clearly,
  \begin{align}
    M_p(R) \setminus \partial M_p(R)  &\subseteq  \Bcal{p}{R}  \subseteq  M_p (R) .
  \end{align}
  If a tile, $J = J(\gamma, \gamma^*)$, contains a point inside the ball, then either $\gamma$ is inside the ball, or $J$ is a boundary tile.
  Indeed, for every point $x \in \gamma \setminus B(R)$, there is a direction $L$, such that $L+x$ intersects $\gamma^*$, hence $(x,L) \in J(\gamma, \gamma^*)$.
  {In other words, if $\gamma$ is not contained in $B(R)$, then neither is the projection of $J$ to $\Rspace^d$.}
  By Lemma~\ref{lem:volume_of_tile}, the measure of this tile
  is $\Vol{p}{\gamma} \Vol{d-p}{\gamma^*} / \tbinom{d}{p}$ and,
  by the mixed regularity property,
  the measures of the tiles corresponding to Delaunay cells inside the ball
  sum up to $\Vol{d}{B(R)}(1 + o(1))$.
  Multiplying by $\tbinom{d}{p}$ completes the proof.
\end{proof}

\section{Average and Expected Distortion}
\label{sec:6}

For a locally finite $A \subseteq \Rspace$, a generically placed $p$-dimensional set $\Omega \subseteq \Rspace^d$ intersects only $(d-p)$-dimensional Voronoi cells of $A$, and any such intersection has a finite \emph{multiplicity}.
In this case, the Voronoi scape of $\Omega$ and $A$, denoted $\cScape{\Omega}{A}$, is the multiset of Delaunay $p$-cells,
in which every $\gamma \in \Delaunay{A}$ appears as many times, as
$\Omega$ intersects its dual $\gamma^*$.
For completeness we mention that for a non-generically placed $A$, the multiplicity can be defined as the (potentially infinite) Euler characteristic of the intersection, and the Voronoi scape can contain Delaunay cells of dimensions different from $p$.
For our analysis these zero-measure set of placements are however irrelevant.
We are ready to prove the main result of the paper.
\begin{theorem}[Average Volume]
  \label{thm:average_volume}
  Let $A \subseteq \Rspace^d$ have the mixed regularity property,
  and let $\Omega$ be a $p$-dimensional rectifiable set in $\Rspace^d$.
  The average $p$-dimensional volume of $\cScape{\Omega}{A}$, averaged over all congruent copies of $\Omega$ inside the $d$-ball with radius $R$ centered at the origin, is $\Vol{p}{\Omega} (\DConstant{p}{d} + o(1))$ as $R$ goes to infinity.
\end{theorem}
\begin{proof}
  We start with the Crofton formula \cite[Formula (5.7)]{ScWe08},
  {which states that the $p$-dimensional volume of $\Omega$ is a constant times the integral of crossings between $\Omega$ and a $(d-p)$-plane, and this constant is the $1$-st projection moment:}
  \begin{align}
    \int_{Q \in \Grass{d-p}{d}} \int_{y \in Q^\bot}
      \chi((Q + y) \cap \Omega) \diff y \diff Q
    &= \moment{p}{d}{1} \Vol{p}{\Omega}.
  \end{align}
  Here $\chi((Q+y) \cap \Omega)$ is the multiplicity of the intersection between $Q+y$ and $\Omega$, which is almost always finite; see \cite[3.16]{Mor16} for the general statement that applies to rectifiable sets.
  Next consider a bounded convex polyhedron, $P$, whose dimension is $d-p$.
  Applying a rigid motion (a rotation composed with a translation), we get a \emph{congruent} copy, $P' \cong P$.
  We represent $P'$ as a polyhedron $P''$ in $Q \in \Grass{d-p}{d}$ and a shift $y \in  Q^\bot$. 
  For any fixed $p$-plane $Q+y$ and any fixed point inside it, the total measure of the congruent copies of $P$ inside $Q+y$ that contain this fixed point is $\Vol{d-p}{P}$. 
  We can thus compute the total measure of intersection points over all congruent copies of $P$ as
  \begin{align}
    \int_{P' \cong P} \chi(P' \cap \Omega) \diff P'
      &= \int_{Q \in \Grass{d-p}{d}}
         \int_{y \in Q^\bot}
         \int_{P \cong P'' \subseteq Q}
          \chi((P''+y) \cap \Omega)
        \diff P'' \diff y \diff Q \\
      &= \int_{Q \in \Grass{d-p}{d}}
         \int_{y \in Q^\bot}
          \Vol{d-p}{P} \chi((Q+y) \cap \Omega) \diff y \diff Q \\
      &= \Vol{d-p}{P} \Vol{p}{\Omega} \moment{p}{d}{1}.
  \end{align}
  Taking $P = \gamma^*$ and moving $\Omega$ instead of the polyhedron, we see that the total measure of intersection points of congruent copies of $\Omega$ with $\gamma^*$ is $\Vol{d-p}{\gamma^*} \Vol{p}{\Omega} \moment{p}{d}{1}$.

  \smallskip
  A $p$-cell $\gamma \in \Delaunay{A}$ belongs to the Voronoi scape of a congruent copy $\Omega'$ of $\Omega$ precisely $\chi(\Omega' \cap \gamma^*)$ times, and we just computed this quantity.
  The total contribution of $\gamma$ to the $p$-dimensional volume of the Voronoi scapes of the congruent copies of $\Omega$ is therefore $\Vol{p}{\gamma} \Vol{d-p}{\gamma^*} \Vol{p}{\Omega} \moment{p}{d}{1}$.
  We get the final result by dividing the total contribution of the $p$-cells in $\Delaunay{A}$ inside $B(R)$ by the total measure
  of the congruent copies inside the ball:
  \begin{align}
    \frac{ \sum_\gamma \Vol{p}{\gamma} \Vol{d-p}{\gamma^*} \Vol{p}{\Omega} \moment{p}{d}{1}}{\Vol{d}{B(R)} (1 + o(1))}
      &= \frac{\binom{d}{p} \Vol{d}{B(R)} (1+o(1)) \Vol{p}{\Omega} \moment{p}{d}{1}}{\Vol{d}{B(R)} (1+o(1))}
        \label{eqn:result1} \\
      &= \Vol{p}{\Omega} (\DConstant{p}{d} + o(1)) ,
        \label{eqn:result2}
  \end{align}
  {in which we use Corollary~\ref{cor:mixed_volumes} to get the right-hand side of \eqref{eqn:result1}, and \eqref{eqn:moment2} to get \eqref{eqn:result2}.}
\end{proof}

We finish by stating the answer to the original question that motivated the work reported in this paper.
We showed in Section~\ref{sec:4} that the stationary Poisson point process has the mixed regularity property in expectation, which allows us to repeat all results while adding the expectation to all quantities.
By the isometry invariance of the process, for any set $\Omega$, the expected volume of $\cScape{\Omega}{A}$ does not depend on the position of $\Omega$.
Exchanging the expectation and the average inside the ball of radius $R$ centered at the origin and letting $R$ go to infinity, we arrive at probabilistic versions of Theorem~\ref{thm:average_volume}:
\begin{theorem}[Expected Volume]
  \label{thm:expected_volume}
  Let $A \subseteq \Rspace^d$ be a stationary Poisson point process with intensity $\rho > 0$, and let $\Omega$ be a compact rectifyable $p$-manifold in $\Rspace^d$.
  Then the expected $p$-dimensional volume of the Voronoi scape of $\Omega$ and $A$ is $\DConstant{p}{d} \Vol{p}{\Omega}$.
\end{theorem}
{Note that the expected volume of the Voronoi scape does not depend on the intensity of the Poisson point process.
On the other hand, the variance does, but this is beyond the scope of this paper.}

\section{Discussion}
\label{sec:8}

The main contribution of this paper is a complete analysis of the average and expected distortion of $p$-dimensional Voronoi scapes in $\Rspace^d$, for $0 \leq p \leq d$.
For $p = 1$, these scapes are known as Voronoi paths, for which the expected distortion has been studied but was known only in $\Rspace^2$; see \cite{BTZ00}.
A useful insight from our analysis is that the expected distortion for a stationary Poisson point process is the average distortion for a general locally finite point set.
We make crucial use of this insight in the proof of our results.
\emph{Can these results be extended to other measures, such as notions of curvature, for example?}
The proof of Theorem~\ref{thm:average_volume} suggests that this extension would require a detailed analysis of the Crofton formula.
Insights in this direction could be helpful in using the Voronoi scape to measure otherwise difficult to measure shapes.

\smallskip
In our analysis, the properties that make a mosaic a Delaunay mosaic
are not used other than in the quantification of the mixed regularity
property for locally finite sets.
Indeed, we only need a pair of dual complexes in which dual cells
are orthogonal to each other,
a property that holds also for the generalizations of Voronoi
tessellations and Delaunay mosaics to points with real weights;
see e.g.\ \cite{Aur87b}.

\subsection*{Acknowledgements}

{\footnotesize
  {The authors thank Ranita Biswas and Tatiana Ezubova for the collaboration on computational experiments that motivated the work reported in this paper.}
  The authors also thank Daniel Bonnema for proofreading and noticing an issue with the original proof of Lemma~\ref{lem:mixed_regularity_in_expectation}.
}



\end{document}